\documentclass[11pt,twoside,draft,reqno]{amsart}

\usepackage{amssymb}

\def\const{\text{\rm const}}

\def\supp{\text{\rm supp}\,}
\def\to{\rightarrow}

\def\bs{\bigskip}\def
\ms{\medskip}
\def\no{\noindent}

\def\PP{{\mathcal{P}}}
\def\QQ{{\mathcal{Q}}}

\def\K{{\mathcal K}}

\def\Z{{\mathbb{Z}}}
\def\N{{\mathbb{N}}}
\def\R{{\mathbb R}}
\def\C{{\mathbb{C}}}

\def\P{\mathcal{P}}

\def\e{\varepsilon}

\def\L{\Lambda}
\def\l{\lambda}
\def\G{\Gamma}
\def\lan{\lambda_n}
\def\Res{\text{\rm Res}}

\def\ti{\tilde}
\def\sm{\setminus}

\def\sign{\text{\rm sign}}

\newtheorem{theorem}{Theorem}[section]
\newtheorem{lemma}[theorem]{Lemma}
\newtheorem{corollary}[theorem]{corollary}
\newtheorem{proposition}[theorem]{Proposition}

\newtheorem{remark}[theorem]{Remark}

\numberwithin{equation}{section}

\author{Alexei Poltoratski}
\address{Texas A\&M University
\\ Department of Mathematics\\
College Station, TX 77843, USA }
\email{alexeip@math.tamu.edu}
\thanks{The author is supported by
N.S.F. grant DMS-1101278}

\title{Bernstein's problem on weighted \linebreak polynomial approximation}

\begin{document}

\begin{abstract}
We formulate  a necessary and sufficient condition for polynomials to be dense
in a space of continuous functions on the real line, with respect to Bernstein's
weighted uniform norm. Equivalently, for a positive finite measure $\mu$ on the real line we give a criterion
for density of polynomials in $L^p(\mu)$.
\end{abstract}

\maketitle

\ms
\section{\bf Introduction}


\ms\no Let $W:\R\to [1,\infty)$ be a continuous function satisfying $x^n=o(W(x))$ for any $n\in\N$, as $x\to\pm\infty$.
Denote by $C_W$ the space of all continuous functions $f$ on $\R$ such that $f/W\to 0$ as $x\to\pm\infty$ with the norm
\begin{equation}||f||_W=\sup_\R\frac{|f|}W.\label{norm}\end{equation}
The famous weighted approximation problem posted by Sergei Bernstein in 1924 \cite{Bernstein}
asks to describe the weights $W$ such that polynomials are dense in $C_W$.

\ms\no Throughout the 20th century Bernstein's problem was investigated by many prominent analysts including N. Akhiezer, L. de Branges, L. Carleson, T. Hall, P. Koosis, B. Levin, P. Malliavin, S. Mandelbrojt, S. Mergelyan, H. Pollard and M. Riesz. This activity continues to our day with more recent significant contributions
by A. Bakan,  M. Benedicks, A. Borichev, P. Koosis, M. Sodin and P. Yuditski, among others.
Besides the natural beauty of the original question, such an extensive interest towards Bernstein's problem is generated by numerous links with adjacent fields, including its close relation with the moment problem.

\ms\no
Further information and references on the remarkable history of Bernstein's problem can be found
in two classical surveys by Akhiezer \cite{Akhiezer} and Mergelyan \cite{Mergelyan}, a recent one by  Lubinsky \cite{Lubinsky}, or
in the first volume of Koosis' book \cite{Koosis}.
Despite a number of deep and beautiful results a complete solution is yet to be found.

\ms\no Most of the results on Bernstein's problem belong to one of the two following groups. The first group, containing classical theorems by Akhiezer, Mergelyan and Pollard as well as more
recent results by Koosis, provides conditions on $W$ in terms of the norms of point evaluation functionals. The second group uses the approach pioneered by
de Branges (see \cite{dBr1} or theorem 66 in \cite{dBr}) and further developed by Borichev, Sodin and Yuditski. These results are formulated in terms of existence of entire
functions belonging to certain classes.

\ms\no Both approaches have produced significant progress towards a full solution, although the conditions of density
remained rather implicit. Besides specific examples, the only general explicit results in the literature are a classical theorem by Hall \cite{Hall} and
a theorem on log-convex weights published by Carleson \cite{Carleson}, see section \ref{exandcor}.

\ms\no In the present paper we start by following the second approach mentioned above. We prove a version of de Branges' theorem that claims
existence of extreme annihilating measures. The novelty of the paper is an additional computational step that allows us to make the final statement
more elementary and at the same time more general.
At that stage we utilize the Titchmarsh-Ulyanov theory of $A$-integrals together with some of
the ideas used by N. Makarov and the author in \cite{MIF} and \cite{MIF2}.

\ms\no The main result of the paper is theorem \ref{main} in section \ref{mainresult}. The statement involves the notion of characteristic
sequences introduced in section \ref{char}. In the last part of the paper we discuss relations of theorem \ref{main}
with  a classical result on log-convex weights and a more recent theorem by Borichev and Sodin. To approach the latter, we give a description of zero
sets of Hamburger entire functions and Krein entire functions of zero exponential type.

\bs The contents of the paper are as follows:
\bs

\begin{itemize}

\item[\bf \ref{C_W}:] The definition and duality of $C_W$ in the case of semi-continuous weights
\item [\bf \ref{Cauchy}:] Definitions and notations  for Cauchy integrals
\item [\bf \ref{char}:] Characteristic sequences
\item [\bf \ref{sectionBakan}:] A result of Bakan on equivalence of approximation
in $L^p$ to Bernstein's problem
\item [\bf \ref{mainresult}:] The statement and  discussion of the main result
\item [\bf \ref{Clark}:] Basics of Clark's theory
\item [\bf \ref{decay}:] The relation between asymptotic decay of a Cauchy integral and annihilation of polynomials
by the corresponding measure
\item [\bf \ref{section66}:] A version of de Branges' theorem 66
\item [\bf \ref{TUT}:] Titchmarsh-Ulyanov's theory of $A$-integrals
\item [\bf \ref{sectionmasses}:] Point masses of extreme measures via  $A$-integrals
\item [\bf \ref{proofs}:] Main proofs
\item [\bf \ref{logconv}:] A classical theorem on log-convex weights
\item [\bf \ref{HK}:] A description of zero sets of Hamburger and Krein entire functions
\item [\bf \ref{sectionBS}:] A result by Borichev and Sodin
\item [\bf \ref{charass}:] Asymptotics of characteristic sequences and further applications

\end{itemize}

\bs\no\textbf{Acknowledgement.} I am grateful to Nikolai Makarov and Misha Sodin for introducing me to the general area of Bernstein's problem and for numerous
thought-provoking discussions.

\ms\section{Preliminaries}\label{Preliminaries}

\subsection{Semi-continuous weights}\label{C_W}
In this paper we allow the weight function $W$ to be semi-continuous from below instead of continuous as in most classical papers.
Throughout the rest of the paper we use the following definition.

\ms\no We say that
a  function $W\geqslant 1$ on $\R$ is a \textit{weight} if $W$ is lower semi-continuous and
$x^n=o(W)$ as $|x|\to\infty$.

\ms\no Our weights are also allowed to take infinite values at finite points on $\R$, which makes it possible to study approximation on subsets of the line within
the same general formulation of the problem.  For instance, the classical Weierstrass theorem
answers the question of density of polynomials in $C_W$ with $W$ equal to 1 on an interval and infinity
elsewhere. Another important case of the problem is approximation on discrete sequences (see, for instance, \cite{BS}), which corresponds to the
weights that are infinite outside of a discrete sequence.

\ms\no With a semi-continuous and $\hat\R$-valued $W$ the quantity $||f||_W$, defined as in the introduction, ceases being a norm and becomes a semi-norm. The set
 of continuous functions $g$ such that $g/W\to 0$ at $\pm\infty$ is no-longer complete.

 \ms\no
The semi-norm defined by \eqref{norm} can be made a norm following a standard
procedure. First the space of continuous functions $g$, such that $g/W\to 0$ at $\pm\infty$, needs to be factorized to obtain a space of equivalence classes: $f\thicksim g$ if and only if $||f-g||_W=0$.
After that the factor-space needs to be completed. We denote by $C_W$ the resulting space.

\ms\no Note that if $W$ is continuous and takes only finite values,
$C_W$ coincides with the space of continuous functions defined in the introduction. In the general case, we still have the following property.

\ms\no If $W$ is a weight we say that a measure $\mu$ on $\R$ is $W$-finite if
$$\int Wd|\mu|<\infty.$$

\begin{proposition}
The dual space of $C_W$ consists of $W$-finite measures.
\end{proposition}

\begin{proof}
Consider a sequence of continuous weights $W_n$ such that $W_{n+1}(x)\geqslant W_n(x)$  and $W_n(x)\to W(x)$ for any $x\in\R$.
Note that any bounded linear functional $\mu$ on $C_W$ induces a linear bounded functional on $C_{W_n}$ for any $n$.
Because of monotonicity, $C_{W_n}\subset C_{W_{n+1}}$.
Since any linear bounded functional on $C_{W_n}$ can be identified with a $W_n$-finite measure, again using monotonicity
of $W_n$, we conclude that $\mu$ can be identified with a $W$-finite measure on the set $\cup C_{W_n}$. Since the last
set is dense in $C_W$ (or, more precisely, the set of equivalence classes, containing the elements from $\cup C_{W_n}$,
is dense in $C_W$), $\mu$ can be identified with a $W$-finite measure on the whole $C_W$.

\end{proof}

\ms\no Note that in the general case of semi-continuous $\hat\R$-valued weights, when we say that polynomials are
not dense in $C_W$ that statement  still means  that there exists a continuous $g$ and $\e>0$ such that $ g/W\to 0$ at $\pm\infty$ and
$||g-p||_W>\e$ for every polynomial.
The crucial dual statement, that characterizes non-completeness in the case of continuous weights, still holds for general $W$:
Polynomials are not dense in $C_W$ if and only if there exists a non-zero $W$-finite measure that annihilates polynomials.

\ms\no For the rest of the paper the reader has a choice: to think of $C_W$ as of a semi-normed space of continuous functions,
or as a completed normed space of equivalence classes, described above. This choice will affect neither the statements nor the proofs.

\subsection{Cauchy integrals}\label{Cauchy}
If $\mu$ is a finite measure on $\R$ we denote by $K\mu$ its Cauchy integral
$$K\mu(z)=\frac 1\pi \int \frac 1{t-z}d\mu(t)$$
defined for all $z\not\in \supp\mu$.

\ms\no We denote by $\Pi$ the Poisson measure on $\R$,
$$d\Pi=\frac{dx}{1+x^2}.$$

\ms\no We say that a measure $\mu$ on $\R$ is Poisson-finite if
$$\int\frac{d|\mu|(t)}{1+t^2}<\infty.$$
For the class of Poisson-finite measures we will need a slightly different Cauchy integral:
$$\K\mu(z)=\frac 1\pi \int \left[\frac 1{t-z}-\frac t{1+t^2}\right]d\mu(t).$$

\subsection{Characteristic sequences}\label{char}
We call a real sequence \textit{discrete} if it does not have finite accumulation points. To simplify the definitions we will always assume
that a discrete sequence is infinite and does not have multiple points. A discrete sequence is called one-sided if it is bounded from below or from above
and two-sided otherwise.

\ms\no If $\L=\{\lan\}$ is a discrete sequence we will always assume that it is \textit{enumerated in the natural order}, i.e.
$\lan<\l_{n+1}$, non-negative elements are indexed with non-negative integers and negative
elements with negative integers.

\ms\no For instance, if $\L=\{\lan\}_{n\in\Z}$ is a two sided sequence then
$$...\l_{-n-1}<\l_{-n}<...<\l_{-1}<0\leqslant \l_0<\l_1<...\lan<\l_{n+1}<...$$
Thus a one-sided sequence bounded from below (above) will be enumerated with $n\in\Z, n\geqslant-N$ ($n\in\Z, n<N$), where $N$
is the number of negative (non-negative) elements in the sequence.

\ms\no We  say that a sequence $\L=\{\lan\}$ has upper density $d$ if
$$\limsup_{A\to\infty}\frac{\#[\L\cap (-A,A)]}{2A}=d.$$
If $d=0$ we say that the sequence has zero density.

\ms\no In the remainder of this section we give key definitions used in the main result.

\ms\no A discrete sequence $\L=\{\lan\}$ is called \textit{balanced} if the limit
\begin{equation}
\lim_{N\to\infty}\sum_{|n|<N}\frac {\l_n}{1+\lan^2}
\label{balance}
\end{equation}
exists.

\ms\no Observe that any even sequence (any sequence $\L$ satisfying $-\L=\L$) is balanced.
So is any two-sided sequence sufficiently close to even.
At the same time, a one-sided sequence
has to tend to infinity fast enough to be balanced (the series $\sum \lan^{-1}$ must converge).

\ms\no Let $\L=\{\lan\}$ be a balanced sequence of finite upper density. For each $n, \ \lan\in\L,$ put
$$p_n= \frac 12\left[\log(1+\l_n^2)+\sum_{n\neq k, \ \l_k\in\L}\log\frac{1+\l_k^2}{(\l_k-\lan)^2}\right],$$
where the sum is understood in the sense of principle value, i.e. as
$$\lim_{N\to\infty}\sum_{0<|n-k|<N}\log\frac{1+\l_k^2}{(\l_k-\lan)^2}.$$
We will call the sequence of such numbers $P=\{p_n\}$ the \textit{characteristic sequence} of $\L$.

\ms\no Note that for a sequence of finite upper density the last limit exists for every $n$ if and only if
it exists for some $n$ \textit{if and only if the sequence is balanced}.

\section{Weighted polynomial approximation}\label{midsection}

\subsection{Equivalence between weighted uniform and $L^p$-approximation}\label{sectionBakan}
Close connections between  $L^p$- and weighted uniform approximation
have been known to the experts for a long time. Nevertheless, the formal result that
 reduces the problem of polynomial approximation in $L^p$-spaces to Bernstein's problem
was found by A. Bakan only recently. This result allows us to concentrate on the latter problem for the rest of the paper.

\begin{theorem}\cite{Bakan}\label{tBakan}
Let  $0< p <\infty $ be a constant and let $\mu$ be a positive finite measure on $\R$ such that $L^p(\mu)$ contains all polynomials. Polynomials are dense in $L^p(\mu)$
if and only if $\mu$ can be represented as $\mu=W^{-p}\nu$ for some finite positive measure $\nu$ and a weight $W$ such that
polynomials are dense in $C_W$.
\end{theorem}

\ms\no Let us point out that the weights appearing in the theorem are lower semi-continuous. Hence, to study the $L^p$- and uniform versions as one problem
 one needs the general definition of $C_W$ discussed above. For reader's convenience we supply a short proof of Bakan's result.


\begin{proof}
If polynomials are dense in $C_W$ for some weight $W$ such that $\mu=W^{-p}\nu$
then for any bounded continuous function $f$ there
exists a sequence of polynomials $\{s_n\}$
such that $s_n/W$ converges to $f/W$ uniformly. Then
$$\int |f-s_n|^p d\mu=\int \frac{|f-s_n|^p}{W^p}W^p d\mu=\int \left|\frac fW-\frac{s_n}W\right|^p d\nu\to 0.$$
Hence polynomials are dense in $L^p(\mu)$.

\ms\no Suppose that polynomials are dense in $L^p(\mu)$. Let $\{f_n\}_{n\in\N}$ be a  set
 of bounded continuous functions on $\R$, that is dense
 in any $C_W$ (one could, for instance, choose a countable set of compactly supported functions,
 dense in any space of continuous functions on a closed finite interval). Let $\{s_{n,k}\}_{n,k\in\N}$
be a family of polynomials such
that
$$||f_n-s_{n, k}||_{L^p(\mu)}<4^{-(n+k)}.$$
Define
$$W=1+\sum_{n,k\in \N}2^{n+k}|f_n-s_{n, k}|.$$
Notice that then  $W\in L^p(\mu)$, $W$ is lower semi-continuous and $s_{n, k}/W\to f_n/W$ uniformly as $k\to\infty$.
Without loss of generality, $L^p(\mu)$ is not finite dimensional. Then $\{s_{n,k}\}$ contains polynomials of arbitrarily
large degrees and $x^n=o(W)$ for any $n$. Thus $W$ is a weight.
Since $\{f_n\}$ is dense in $C_W$, polynomials are dense in $C_W$. The measure $\nu$ can be chosen as $W^p\mu$.

\end{proof}

\subsection{Main result}\label{mainresult}
In regard to the problem of uniform approximation,  we intend to prove the following result.
Recall that per our agreement all sequences  are assumed to be infinite.
The definition and formula for the characteristic sequence of a balanced sequence was given in section~\ref{char}.

\begin{theorem}\label{main}
Polynomials are not dense in $C_W$ if and only if there exists a balanced sequence $\L=\{\lan\}$
of zero density such that $\L$ and its characteristic sequence $P=\{p_n\}$ satisfy
\begin{equation}
\sum W(\lan)\exp (p_n)<\infty.
\label{condition}
\end{equation}
\end{theorem}

\ms\no In the rest of this section let us discuss some reformulations of the above result.

\begin{remark}\label{x^n}
We will call a weight degenerate if the set $\{W<\infty\}$ is a discrete sequence.
All other weights will be called non-degenerate. Degenerate weights are used
in problems of weighted approximation at discrete sequences of points.

\ms\no Note that one can add or remove finitely many points from $\L$ to change $p_n$
by $C\log |\lan|$. It follows that
if one allows only non-degenerate weights in theorem \ref{main} then the condition \eqref{condition} can be simplified to
$$ \log W(\lan)\leqslant -p_n$$
for large enough $n$ or to
\begin{equation}
\log W(\lan)\leqslant -p_n+O(\log |\lan|),
\label{conditionlog}
\end{equation}
depending on the direction the result needs to be applied in.

\ms\no Hence, for any non-degenerate weight $W$ and $n\in\N$, polynomials are dense in $C_W$ if and only if they are dense
in $C_{(1+|x|^n)W}$. This statement can also be obtained from lemma \ref{growth}.
\end{remark}

\begin{remark}\label{even}
A case often treated in classical literature is the case of even weights. If $W$ is even, then polynomials are not dense in $C_W$
if and only if there exists an even sequence $\L$ like in theorem \ref{main}.
Indeed, in that case there exists an even $W$-finite measure $\mu$ annihilating polynomials.
The rest follows from the last statement of lemma \ref{t66} below.
\end{remark}

\ms\no We say that a measure is discrete if is supported on a discrete sequence.
An important case of the problem of density of polynomials in $L^p(\mu)$ is when $\mu$ is discrete.
In that situation theorem \ref{main} yields the following statement (e.g. theorem A in \cite{BS} discussed in
section \ref{sectionBS}).

\begin{corollary}\label{discrete}
Let $\mu$ be a finite discrete measure supported on $\L=\{\lan\}$ such that $L^1(\mu)$ contains polynomials.

\ms\no I) Polynomials are not dense in $L^p(\mu)$ for some $1<p<\infty$ if and only if there exists a balanced zero density subsequence $\G=\{\gamma_n\}\subset\L$  such that the characteristic sequence  $P=\{p_n\}$ of $\G$ satisfies
$$\sum_{\G}[\mu(\{\gamma_n\})]^{1-q}\exp{( q p_n)}<\infty,$$
where $\frac 1p +\frac 1q =1$.

\ms\no II)  Polynomials are not dense in $L^1(\mu)$  if and only if there exists a balanced zero density subsequence $\G=\{\gamma_n\}\subset\L$ such that the characteristic sequence $P=\{p_n\}$ of $\G$ satisfies
$$\exp{p_n}=O(\mu(\{\gamma_n\}))$$
as $|n|\to \infty$.

\ms\no III) Define the weight $W$ as $W(\lan)=1+[\mu(\{\l_n\})]^{-1}$ and $W(x)=\infty$ for all $x\not\in \L$. Polynomials are not dense in $C_W$  if and only if there exists a balanced zero density subsequence $\G\subset\L$ such that the characteristic sequence  $P=\{p_n\}$ of $\G$ satisfies
$$\sum_{\G} \frac{\exp{  p_n}}{\mu(\{\gamma_n\})}<\infty.$$
\end{corollary}

\ms\no We postpone the proofs until section \ref{proofs}.

\section{Lemmas and proofs}

\subsection{Basics of Clark theory}\label{Clark}

\ms\no By $H^2$ we denote the Hardy space in the upper half-plane
$\C_+$. We say that an inner function
$\theta(z)$ in $\C_+$ is  meromorphic if it allows a
meromorphic extension to the whole complex plane. The meromorphic
extension to the lower half-plane $\C_{-}$ is given by
$$\theta(z)=\frac{1}{\theta^{\#}(z)} $$
where $\theta^{\#}(z)=\bar\theta(\bar z)$.

\ms\no Each inner function $\theta(z)$ determines a model subspace
$$K_\theta=H^2\ominus \theta H^2$$ of the Hardy space
$H^2(\C_+)$. These subspaces play an important role in complex and
harmonic analysis, as well as in operator theory,
see~\cite{Ni2}.

\ms\no For each inner function $\theta(z)$ one can consider a positive harmonic
function
$$\Re \frac{1+\theta(z)}{1-\theta(z)}$$
 and, by the Herglotz
representation, a positive measure $\mu$ such that
\begin{equation} \label{for1} \Re
\frac{1+\theta(z)}{1-\theta(z)}=py+\frac{1}{\pi}\int{\frac{yd\mu
(t)}{(x-t)^2+y^2}}, \hspace{1cm} z=x+iy,\end{equation} for some $p
\geqslant 0$. The number $p$ can be viewed as a point mass at infinity.
The measure $\mu$ is Poisson-finite, singular and supported on the set where non-tangential limits
of $\theta$ are equal to $1$.
The measure
$\mu +p\delta_\infty$ on $\hat\R$
is called the Clark measure for $\theta(z)$.

\ms\no The Clark measure defined in \eqref{for1} is often denoted by $\mu_1$.
If $\alpha\in \C, |\alpha|=1$ then $\mu_\alpha$ is the measure defined by \eqref{for1}
with $\theta$ replaced by $\bar\alpha\theta$.

\ms\no Conversely, for
every positive singular Poisson-finite measure $\mu$  and a number $p \geqslant
0$, there exists an inner function $\theta(z)$ satisfying~\eqref{for1}.

\ms\no Every function $f \in K_\theta$ can be represented by the formula
\begin{equation} \label{for2} f(z)=\frac{1-\theta(z)}{2\pi
i}\left(p\int{f(t)\overline{(1-\theta(t))}dt}+\pi Kf\mu(z)\right).
\end{equation}
If the Clark measure does not have a point mass at infinity,
the formula is simplified to
\begin{equation}2 i f(z)=(1-\theta(z))Kf\mu(z).\label{for2a}\end{equation}
These formulas define a unitary operator from
$L^2(\mu)$ to $K_\theta$.
Similar formulas can be written for any $\mu_\alpha$ corresponding to $\theta$.
For any $\alpha,\ |\alpha|=1$ and any $f\in K_\theta$, $f$ has non-tangential boundary
values $\mu_\alpha$-a.e. on $\R$. Those boundary values can be used in \eqref{for2} or \eqref{for2a} to
recover $f$.

\ms\no In the case of meromorphic $\theta(z)$,
every function $f \in K_\theta$ also has a meromorphic extension in
$\C$ given by the formula~\eqref{for2}. The
corresponding Clark measure is discrete  with atoms at the points
of $\{\theta=1\}\subset \hat\R$ of the size
$$\mu(\{x\})=\frac{2\pi
}{|\theta'(x)|}.$$
If $\L\subset \R\ (\hat\R)$ is a given discrete sequence, one can easily construct a meromorphic
inner function $\theta$ satisfying $\{\theta=1\}=\L$ by  considering a positive
Poisson-finite measure concentrated on $\L$ and then choosing $\theta$ to satisfy
\eqref{for1}. One can prescribe the derivatives of $\theta$ at $\L$ with a proper choice
of pointmasses.

\ms\no For more details on Clark measures and further references the reader may consult \cite{PS}.

\subsection{Polynomial annihilation and asymptotic decay}\label{decay}
 Suppose that $L^1(|\mu|)$ contains all polynomials. We say that
$\mu$ annihilates polynomials (and occasionally write $\mu\perp\P$) if
$$\int x^nd\mu(x)=0$$
for all $n\in\Z,\ n\geqslant 0$.

\begin{lemma}\label{growth}
A measure $\mu$ with finite moments annihilates polynomials if and only if
$$K\mu(iy)=o(y^{-n})$$
for any $n>0$ as $y\to\infty$.
\end{lemma}

\begin{proof} Suppose that $\mu\perp\P$. Since
$(t^n-z^n)/(t-z)$ is a polynomial of $t$ for every fixed $z$,
$$0=\int\frac{t^n-z^n}{t-z}d\mu(t)=[Kt^n\mu](z)-z^nK\mu(z).$$
Since any Cauchy integral of a finite measure tends to zero along $i\R_+$, so does $Kt^n\mu$. Hence $K\mu(z)=o(z^{-n})$
as $z\to\infty,\ z\in i\R_+$.

\ms\no Conversely, suppose that $K\mu(iy)=o(y^{-n})$
for any $n>0$ as $y\to\infty$.
Without loss of generality, $\mu$ is real (otherwise consider $\mu-\bar\mu$ or $i(\mu+\bar\mu)$).
Then $$K\mu(-iy)=\overline{K\mu(iy)}=o(y^{-n})$$ as well.
Since $\mu$ has finite moments we may consider the function
$$H(z)=\int\frac{t^n-z^n}{t-z}d\mu(t).$$
It is easy to show that $H$ is entire of exponential type zero.
Noticing again that
$$H(z)=[Kt^n\mu](z)-z^nK\mu(z),$$
we see that $H$ is bounded on $i\R$. Hence $H$ is a constant by the Phragmen-Lindell\"of principle. Since $H(iy)$ tends to zero, $H$ is zero. Therefore
$$z^nK\mu(z)=[Kt^n\mu](z)=\int\frac{t^n}{t-z}d\mu(t).$$
Putting $z=0$ in the last equation we get that $\mu$ annihilates $t^{n-1}$ for any $n>0$.
\end{proof}

\subsection{A version of de Branges' theorem 66.}\label{section66}
An important tool in the study of completeness problems is a theorem by de Branges that reduces
the problem to a question of existence of an entire function with certain extremal properties.
A version of this theorem applicable to polynomial approximation can be found in \cite{dBr1}.
Another version, pertaining to exponential approximation, is theorem 66 in \cite{dBr}.
Further variations of this result, along with a detailed discussion of applications can be found
in \cite{BS, S, SY}.

\ms\no The theorem can be equivalently reformulated in terms of existence of extremal measures with certain
properties of Cauchy integrals. Statements of that kind were formulated in \cite{GAP, Type}.
In this section we discuss yet another version of that result applicable in our settings.

\begin{lemma}\label{t66} Let $W$ be a weight and let $\mu\neq 0$ be a $W$-finite complex measure on $\R$ that annihilates polynomials.
Then there exists a real finite discrete measure $\nu=\sum \alpha_n\delta_{\lan}$ such that

\ms\no 1) $\supp\nu=\{\lan\}\subset\supp\mu$,

\ms\no 2) $\nu$ is $W$-finite,

\ms\no 3) $\nu\perp\P$,

\ms\no 4) $K\nu\neq 0$ anywhere in $\C$ and is outer in $\C_\pm$.

\ms\no If $\mu$ is even, $\nu$ can be chosen to be even.

\end{lemma}

\begin{proof}
Without loss of generality
$$\int W(x)d|\mu|(x)=1.$$
We can also assume that the measure is real (otherwise consider $\mu\pm\bar\mu$).

\ms\no Denote by $S$ the following set of measures:
$$S=\{\ \nu\ |\ \int Wd|\nu|\leqslant 1, \ \supp\nu\subset\supp\mu,\ \nu\perp\P,\ \nu=\bar\nu\}.$$
Since $\mu\in S$, the set is non-empty. It is also convex and $*$-weakly closed
in the space of all $W$-finite measures.
Therefore by the Krein-Milman theorem it has a non-zero extreme point. Let $\nu$ be such a point.

\ms\no First, let us show that
the set of real
$L^\infty(|\nu|)$-functions $h$, such that $ h\nu\perp\P$, is one-dimensional and  therefore $h=c\in \R$.
(This is equivalent to the statement that the closure of polynomials in $L^1(|\nu|)$ has deficiency 1, i.e. the space
of its annihilators is one dimensional.)

\ms\no Let there be a bounded real $h$ such that $h\nu\perp\P$.
Without loss of generality $h\geqslant 0$, since one can add constants, and $\int Whd|\nu|=1$. Choose
$0<\alpha<1$ so that $0\leqslant\alpha h<1$. Consider the measures
$$\nu_1=h\nu\quad\textrm{ and }\quad\nu_2=(1-\alpha)^{-1}(\nu-\alpha\nu_1).$$ Then both of them
belong to $S$ and $\nu=\alpha \nu_1+(1-\alpha)\nu_2$, which contradicts
the extremity of $\nu$.

\ms\no Now let us show that $\nu$ is discrete. Let $g$ be a continuous compactly supported real function on $\R$
such that $\int gd|\nu|=0$. By the previous part, there exists a sequence  of polynomials $f_n$,
$f_n\to g$ in $L^1(|\nu|)$. Indeed, otherwise there would exist a function $h\in L^\infty(|\nu|)$ annihilating all polynomials  and such that $\int hgd|\nu|=1$. Since $\int gd|\nu|=0$, $h\neq const$ and we would obtain a contradiction
with the property that the space of annihilators is one-dimensional.

\ms\no Since $\nu$ annihilates polynomials and $(f_n(z)-f_n(w))/(z-w)$ is a polynomial for every fixed $w\in\C\sm\R$,
$$0=\int \frac{f_n(z)-f_n(w)}{z-w}d\nu(z)=Kf_n\nu(w)- f_n(w)K\nu(w)$$
and therefore
$$f_n(w)=\frac{Kf_n\nu}{K\nu}(w).$$
Taking the limit,
$$f=\lim f_n=\lim \frac{Kf_n\nu}{K\nu}=\frac{K g\nu}{K\nu}.$$

\ms\no Since all of $f_n$ are polynomials, one can show that the limit function $f$
is entire. Indeed, first notice that there exists a positive function $V\in L^1(|\nu|)$
such that $f_{n_k}/V \to g/V$ in $L^\infty(|\nu|)$, for some subsequence $\{f_{n_k}\}$. To find such a $V$ first choose
$f_{n_k}$ so that
$$||f_{n_k}-g||_{L^1(|\nu|)}<3^{-k}$$
 and then put
$$V=1+\sum 2^k |f_{n_k}-g|.$$
Denote $F_k=f_{n_k}/V$ and $\eta=V|\nu|$. Then $F_k$ converge in $L^2(\eta)$ and by the Clark theorem
$(1-I)KF_k\eta$ converge in $H^2(\C_+)$,
where $I$ is the inner function whose Clark measure is $\eta$. Notice that
$$f_{n_k}=\frac{Kf_{n_k}\nu}{K\nu}=\frac{KF_k\eta}{K\nu}=\frac{(1-I)KF_k\eta}{(1-I)K\nu}.$$
Now let $T$ be a large circle in $\C$ such that $|(1-I)K\nu|>\const>0$ on $T$.
Denote $T_\pm=T\cap \C_\pm$ and let $m_T$ be the Lebesgue measure on $T$.
Since
$(1-I)KF_k\eta$ converge in $H^2(\C_+)$, $f_{n_k}$ converge in $L^1(T_+, m_T)$. Similarly,
$f_{n_k}$ converge in $L^1(T_-, m_T)$. By the Cauchy formula it follows that $f_{n_k}$ converge normally
inside $T$ and therefore $f$ is analytic inside $T$. Since such a circle $T$ can be chosen
to surround any bounded subset of $\C$, $f$ is entire.

\ms\no Since the numerator in the representation
$$f=\frac{K g\nu}{K\nu}$$
is analytic outside the compact support of $g$, the measure in the denominator
must be singular outside of that support: Cauchy integrals of non-singular measures have jumps at the real line on the support of the a.c. part, which would contradict the property that $f$ is entire. Choosing two different functions $g$ with disjoint supports
we conclude that $\nu$ is singular.

\ms\no  Moreover, since $f$ is entire, the zero set of $f$ has to be discrete.
Since $\nu$ is singular, $K\nu$ tends to $\infty$ nontangentially in $\C_+$ at $\nu$-a.e. point and $f=0$ at $\nu$-a.e. point outside of the support of $g$.
Again, by choosing two different $g$ with disjoint supports, we can see that  $\nu$ is concentrated on a discrete set.

\ms\no Next, let us verify 4. Let $J$ be the inner function corresponding to $|\nu|$ ($|\nu|$ is the Clark measure for $J$). Denote
$$G=\frac 1{2 i}(1-J)K\nu\in K_J.$$
As was mentioned in section \ref{Clark}, $G$ has non-tangential boundary values $|\nu|$-a.e. and
$$\nu=G|\nu|.$$
Since $K\nu(iy)$ tends to $0$ sup-polynomially as $y\to\infty$ by lemma \ref{growth},  so does $G(iy)$.
Suppose that  $G=UH$ in $\C_+$ for some inner $U$. Then there exists a proper inner divisor $I$ of $U$ such that
$IH(iy)$ still decays sup-polynomially as $y\to\infty$. Since $IH\in K_J$,
$$IH=\frac 1{2 i}(1-J)K(IH|\nu|).$$
Since $y^{-1}=O(1-J(iy))$ as $y\to\infty$,  $K(IH|\nu|)$ decays sup-polynomially on the upper imaginary half-axis.
Hence by lemma \ref{growth}, the measure $IH|\nu|=\bar U I\nu$ annihilates polynomials, which again contradicts
the property that the space of annihilators is one-dimensional.
Therefore, $K\nu$ is outer in $C_+$. Since $\nu$ is real, $K\nu(\bar z)=\overline{K\nu(z)}$ and $K\nu$ is outer in $C_-$.

\ms\no If $G$ has a zero at $x=a\in \R$ outside of $\supp \nu$ then
$$\frac G{x-a}\in K_J$$
 and the measure
$$\gamma= \frac G{x-a}|\nu|$$
leads to a similar contradiction with the property that the space of annihilators is one-dimensional, since $(x-a)^{-1}$ is bounded and real on the support of $\nu$. Since
$G=\frac 1{2\pi i}(1-J)K\nu,$
 $K\nu$ does not have any  zeros on $\R$.

\ms\no The last statement of the lemma can be proved by choosing the set $S$ above to consist of even measures and repeating the steps.
\end{proof}

\begin{corollary}
Let $W$ be a weight such that polynomials are not dense in $C_W$. Then there exists a discrete measure $\nu$ satisfying
the conditions 1-4 of the last lemma. If $W$ is even, $\nu$ can be chosen to be even.
\end{corollary}

\begin{proof} The statement follows from the lemma via duality. The space dual to $C_W$ is the space of $W$-finite measures $\mu$.
If polynomials are not dense in $C_W$ then the dual space contains a measure $\mu$ annihilating
polynomials and the lemma can be applied. If $W$ is even, $\mu(x)+\mu(-x)$ can be taken instead of $\mu$. Notice that
$\mu(x)+\mu(-x)\neq 0$ because otherwise $\mu$ cannot annihilate monomials with odd powers.
\end{proof}

\subsection{The Titchmarsh-Ulyanov theory of $A$-integrals}\label{TUT} Let $h\in L^1_{{\rm loc}}(\R)$ be a real-valued function. For each $A>0$ we denote
$$h^A=\begin{cases}h(x),\quad |h(x)|\le A,\\
A,\;\qquad h(x)> A, \\
-A,\;\qquad h(x)< -A.\end{cases}$$
The Cauchy $A$-integral of $h$ is defined by the formula
$$\K_{(A)}h(z)=\lim_{A\to\infty}~\K h^A(z),\qquad z\in \C_+,$$
provided that the limit exists for all $z$. Similarly, one may define the Poisson and the conjugate Poisson $A$-integrals $\PP_{(A)}h$ and $\QQ_{(A)}h$ respectively so that
$$\K_{(A)}h=i\PP_{(A)}h-\QQ_{(A)}h.$$

\ms\no
We denote by $\ti h$ the harmonic conjugate of $h$.
Recall that if $h,\ti h\in L^1(\Pi)$, then $\K\ti h=-i\K h+i\K h(i).$ The following well-known theorem allows one to recover $ h$ from $\ti h$
even when $\ti h\not\in L^1(\Pi)$.

\ms\begin{theorem}\label{TiUl} If $h\in L^1(\Pi)$, then the Cauchy $A$-integral of $\ti h$ exists,  and we have
\begin{equation}\label{Ai}\K_{(A)}\ti h(z)=-i\K h(z)+i\K h(i),\qquad z\in \C\setminus\supp \ti h.\end{equation}
\end{theorem}

\ms\no The imaginary part of the equation \eqref{Ai}, or rather its special case
\begin{equation}\label{tit}\PP_{(A)}\ti h(i)=0,\end{equation}is due to Titchmarsh, see \cite{Z}, and the real part of \eqref{Ai},
\begin{equation}\label{ul}\QQ_{(A)}\ti h=-\PP h+\PP h(i),\end{equation}is  Uly'anov's theorem, see \cite{A2} for a shorter proof.

\ms\no The classical results state convergence for all $z\in\C_\pm$, but can be extended to all $z$ outside of
$\supp\ti h$. We will apply theorem \ref{TiUl} in the special case when
 $\ti h(t)$ is monotonically increasing, and therefore $\ti h(t)=o(t)$. In that case such an extension can be obtained via the following simple argument. Let $x\in\R\setminus\supp\ti h$.
Let $\e>0$ and note that
$$C_\e(t)=\frac 1{t-x}-\frac 12\left[\frac 1{t-x-i\e}+\frac 1{t-x+i\e}\right]$$
decays like $t^{-3}$ as $t\to\pm\infty$ and bounded by $\const \times\e^2$ on $\supp\ti h$.
Since
$$\frac 1\pi\int C_\e (t) \ti h(t) dt=\K_{(A)}\ti h(x)-\frac 12\left[\K_{(A)}\ti h(x+i\e)+\K_{(A)}\ti h(x-i\e)\right],$$
and the last two $A$-integrals exist, $\K_{(A)}\ti h(x)$ exists. Tending $\e$  to zero and noticing that
$$\K_{(A)}\ti h(x+i0)+\K_{(A)}\ti h(x-i0)=2h(x),$$
we obtain the desired extension.

\subsection{Masses of extreme measures}\label{sectionmasses}
\ms\no In our settings, the statement on convergence of $A$-integrals becomes the property of existence of  characteristic sequences
for supports of extreme measures that appear in de Branges' theorem (lemma \ref{t66}).

\begin{lemma}\label{masses}
Let  $\nu=\sum \alpha_n\delta_{\lan}$ be a real finite discrete measure that satisfies the last two conditions
of the lemma \ref{t66}, i.e. such that

\ms\no 1) $\nu$ annihilates polynomials and

\ms\no 2) $K\nu\neq 0$ anywhere in $\C$ and is outer in $\C_\pm$.

\ms\no Then $\L=\{\lan\}$ is a balanced sequence of zero density  and
$$\alpha_n=\const (-1)^n\exp (p_n),$$
 where $P=\{p_n\}$ is the characteristic sequence of $\L$.
\end{lemma}

\begin{proof}
Consider the function $l=\log(K\mu)$ in the upper half-plane. Then $l$ can be represented as
$l(z)=iu - \tilde u$, where $u$ is a continuous branch of argument of $K\nu$ in $\C_+$ and $\tilde u=\log |K\nu|$
is a harmonic conjugate of $u$. Notice that $u$ jumps up by $\pi$ at every $\lan$ and is equal to an integer multiple of $\pi$
between the points of $\L$ on $\R$. Without loss of generality $u=\pi n_\L$ where $n_\L$ is a counting function of $\L$.

\ms\no The property that $\L$ has zero density now follows  from the fact that $\tilde u\in L^1(\Pi)$ and therefore
$\Pi(\{u>t\})=o(1/t)$ as $t\to\infty$, by Kolmogorov's weak type estimates.

\ms\no Let now $\L'=\{\lan\}_{n\neq m}$ for some fixed $m$. Put $v=\pi n_{\L'}$, where
$n_{\L'}$ is the counting function of $\L'$ that is equal to $0$ at $\l_m$.
Then $\tilde v\in L^1(\Pi)$.

\ms\no For $A>0$ denote
$$v_A(x)=\begin{cases}
v(x),\textrm{ if }|v(x)|\le A\\
A,\textrm{ if }v(x)>A\\
-A\textrm{ if }v(x)<A
\end{cases}
$$
Since $\l_m\not\in\supp v$, it follows from the Titchmarsh-Ulyanov theory that
\begin{equation}\K v_N(\l_m)\to \ti v(\l_m)+\const\label{TU}\end{equation}
as $N\to\infty$.

\ms\no Elementary calculations show that
$$Kv_N(\l_m)=\sum_{1\leqslant |m-k|\leqslant N}\frac 12\log\frac{(\l_m-\l_k)^2}{1+\l_k^2}.$$
The last sum tends to $p_m-\log(1+\l_m^2)$ as $N\to\infty$ by the definition of the characteristic sequence.

\ms\no Recall that
$$u(x)-v(x)=\frac\pi2(\sign (x-\l_m)+1).$$
Hence
$$[\K(u-v)](t)=-\frac 12\log\frac{(\l_m-t)^2}{1+\l_k^2}$$
and \eqref{TU} implies
$$\log |\alpha_m|=\log|\Res_{\l_m} K\mu|=p_m+\const.$$

\end{proof}

\begin{remark}\label{remmasses}
As follows from the proof of theorem \ref{main} and the discussion in section \ref{HK},
one can formulate the following converse to the last statement. If $\L$ is a balanced sequence
such that its characteristic sequence satisfies
$$\lim_{|n|\to\infty}\frac{\log|\lan|}{p_n}=0,$$
then
the measure $$\mu=\sum (-1)^n\exp(p_n)\delta_{\lan}$$
satisfies the conditions of the last lemma, i.e. it annihilates polynomials and its Cauchy
transform is zero free and outer in $\C_\pm$. The function $1/K\mu$ is the unique, up to a constant multiple, Hamburger
entire function with the zero set $\L$, see section \ref{HK}.

\end{remark}

\subsection{Main proofs}\label{proofs}

\begin{proof}[Proof of theorem \ref{main}]
$$$$
\ms\no I) First, assume that polynomials are not dense in $C_W$. Then there exists a non-zero real $W$-finite measure $\mu$,
$\mu\perp\P$.
 Therefore by lemma \ref{t66} there exists a discrete measure
$\nu$ satisfying the conditions of the lemma. Denote by $\L$ the sequence supporting $\nu$. Since $\nu$ is $W$-finite, we obtain \eqref{condition}
from lemma \ref{masses}.

\bs\bs

\no II) Now suppose that there exists a zero density balanced sequence $\L$ satisfying \eqref{condition}. Since  $W$
grows faster than polynomials, the measure $$\mu=\sum(-1)^n\exp(p_n)\delta_{\lan}   $$
is finite.

\ms\no Let us notice also that the limit
$$F_\L(z)=\lim_{N\to\infty}(-1)^N\prod_{-N}^N\frac{\sqrt{1+\lan^2}}{z-\lan}$$
exists for any $z\not\in \L$ and defines a non-vanishing analytic function in $\C\sm\L$.
This follows from the observation that every partial product satisfies
$$\log \left|\prod_{-N}^N\frac{\sqrt{1+\lan^2}}{z-\lan}\right|=\sum_{-N}^N\log\frac{\sqrt{1+\lan^2}}{|z-\lan|}$$
and the property that $\L$ has zero density and is balanced.
The function $F_\L$ has simple poles at the points of $\L$ satisfying
\begin{equation}\log |\Res_{\lan}F_\L|=p_n+C\label{pms}\end{equation}
and decays sup-polynomially along $i\R$.
The argument of $F_L$ is equal to $\pi n_\L$ on $\R$ and therefore the signs of the residues
alternate.
Hence $F_L-e^CK\mu$ is an entire function
of exponential type $0$ that tends to zero along the imaginary axis. It follows
that $F_L-e^CK\mu\equiv 0$ and $F_L=\const K\mu$. Thus the Cauchy integral of $K\mu$
decays sup-polynomially along $i\R_+$ and, by lemma \ref{growth}, $\mu$ annihilates polynomials.
Since $\mu$ is $W$-finite, by duality polynomials are not dense in $C_W$.
\end{proof}

\begin{proof}[Proof of corollary \ref{discrete}]
I) Suppose that polynomials are dense in $L^p(\mu)$ but the sum is finite for some balanced zero density $\G\subset \L$.
By theorem \ref{tBakan}, $\mu=W^{-p}\nu$ for some weight such that polynomials are dense in $C_W$ and some finite measure $\nu$.
Then $W\in L^p(\mu)$. Note that the function $\phi$ defined as $\phi(\gamma_n)=\exp (p_n)/\mu(\{\gamma_n\})$
on $\G$ and $\phi(\lan)=0$ for $\lan\not\in \G$ belongs to $L^q(\mu)$.
Hence
$$<W,\phi>_\mu=\sum_{\G}  W(\gamma_n)\exp(p_n)<\infty$$
which contradicts theorem \ref{main}.

\ms\no Conversely, suppose that polynomials are not dense in $L^p(\mu)$ but the sum is infinite for all balanced
zero density subsequences of $\L$.
Then $f\mu$ annihilates polynomials for some $f\in L^q(\mu)$. If $Kf\mu$ is non-vanishing
in $\C\setminus\supp f\mu$ and outer in $\C_\pm$, then by lemma \ref{masses} $$(f\mu)(\{\gamma_n\})=\const (-1)^n\exp(p_n),$$
where $\G=\{\gamma_n\}=\supp f\mu$ and $P=\{p_n\}$ is the characteristic sequence of $\G$. Hence the
sum in the statement
is finite for $\G$ because $f\in L^q(\mu)$.

\ms\no If $Kf\mu$ has a zero at some point $a\not\in\supp f\mu$ then
$$\frac{Kf\mu}{z-a}=Kg\mu$$
where $g=f/(z-a)$, see for instance the proof of lemma \ref{growth}. Observe that $g\in L^q(\mu)$ and $g\mu$ still annihilates polynomials
by lemma \ref{growth}. Since $\L$ has density zero and $f\in L^q(\mu)$, the measure
$g\mu$ is $W$-finite for the weight $W$ defined as $W(\lan)=[\mu(\{\lan\})]^{-1/p}$
and as infinity elsewhere.
Hence polynomials are not dense in $C_W$ and by theorem \ref{main}, \eqref{condition} holds for some subsequence $\G$
of $\L$. Since for every $h\in L^p(\mu)$,
$$ |h(\lan)|\leqslant \const[\mu(\{\lan\})]^{-1/p}=\const W(\lan),$$ we have
$$\infty>\sum W(\gamma_n)\exp(p_n)\geqslant\left|\sum_{} h(\gamma_n)\exp(p_n)\right|=
$$$$\left|\sum \mu(\{\gamma_n\}) h(\gamma_n)\frac{\exp(p_n)}{\mu(\{\gamma_n\})}\right|=|<h,\phi>|,$$
where   $p_n$ is the characteristic sequence of $\G$ and $\phi(\gamma_n)=\exp(p_n)/\mu(\{\gamma_n\})$. Hence $\phi\in L^q(\mu)$ which implies that the sum
in I) is finite for $\G$.

\ms\no If $Kf\mu$ is non-vanishing but has a non-trivial inner factor $e^{2iaz},\ a>0$ in $\C_+$ then
$$\frac{Kf\mu}{e^{iaz}}=Kh\mu$$
with $h=e^{-iaz}f$, as follows for instance from theorems 3.3 and 3.4 in \cite{Clark}.
Then the
Cauchy integral $K(f-\frac h2)\mu$ still annihilates polynomials by lemma \ref{growth} and vanishes at
infinitely many points in $\C_+$. Hence one can factor out one of the zeros and repeat the previous argument. The case when $Kf\mu$ is non-outer in $\C_-$ is similar.

\ms\no II) Can be proved in a similar way.

\ms\no III) Follows directly from theorem \ref{main}.

\end{proof}

\section{Examples and corollaries}\label{exandcor}

\ms\no This section contains further discussion of theorem \ref{main} including its relations with some of the known results.

\ms\no A classical theorem by Hall \cite{Hall} says that if
$$\int_{-\infty}^\infty\frac{\log W}{1+x^2}dx<\infty$$
for a weight $W$ then polynomials are not dense in $C_W$.
Indeed, if $F$ is an outer function in $\C_+$ satisfying
$$|F|=\frac 1{(1+x^2)W},$$ then the measure $e^{ix}F(x)dx$ is a $W$-finite
measure that annihilates polynomials by lemma \ref{decay}.

\ms\no A direct inverse to this statement is false. Even if one requires
that $\log W$ is poisson unsummable and $W$ is monotone on $\R_\pm$,
the polynomials may still not be dense in $C_W$, as follows from an example given
in \cite{Koosis}.

\subsection{Log-convex weights}\label{logconv}
We say that
$f:E\subset\R_+\to\R$ is log-convex if it is convex as a function of $\log x$,
i.e. if the function $g(t)=f(e^t)$ is convex on $S=\log E=\{\log x|\ x\in E\}$. In particular, a twice differentiable
function $f$ is log-convex on an interval $(a,b)\subset \R_+$ if $f'(x)+xf''(x)\geqslant 0$ for all $x\in (a,b)$.

\ms\no The following classical result, published by L. Carleson in \cite{Carleson}, but seemingly known earlier to several other mathematicians (see for instance \cite{IK}), is
 a partial inverse to Hall's theorem.

\begin{theorem}\label{Carleson}
Let $W$ be an even weight that is log-convex on $\R_+$. Then polynomials are not dense in $C_W$ if and only if $\log W\in L^1(\Pi)$.

\end{theorem}

\begin{proof}
If $S=\{s_n\}$ is  an even discrete sequence of finite density denote by $v_S$ the function
$$v_S(x)=\frac12\sum \log\left|\frac{(s_n-x)^2}{1+s_n^2}\right|,$$
where the sum is understood in terms of normal convergence of partial sums $\sum_{|n|<N}$ in $\C\setminus\L$. Simple computations show that $-v_S$ is log-convex on every interval $(s_n,s_{n+1}),\ \lan\geqslant 0$.

\ms\no To prove the theorem, notice that in one direction it follows from Hall's result. In the opposite direction, suppose that
polynomials are not dense in $C_W$. Then  there exists a sequence $\L$ like in the statement of  theorem \ref{main}. By remark \ref{even},
$\L$ can be chosen to be even.

\ms\no Fix $n>0$ and denote $\Gamma_n=\L\setminus\{\lan,\l_{-n},\l_{n+1},\l_{-n-1}\}$. Then \eqref{condition} implies
$$\log W(\l_k)\leqslant v_{\Gamma_n}(\l_k) + \frac 12\log_-\frac{(\lan-\l_{n+1})^2}{1+\l_{n+1}^2}+\const,\textrm{ for }k=n,n+1.$$
Since both $W$ and $-v_{\Gamma_n}$ are log-convex on $(\lan,\l_{n+1})$ the inequality can be extended to the whole interval
$(\lan,\l_{n+1})$ for every $n$. Since $v_\L\in L^1(\Pi)$, the quantaty
$$\sum_n\int_{\lan}^{\l_{n+1}}|v_\L-v_{\G_n}|d\Pi$$
is finite  and $\log W\geqslant 0$, this implies that $\log W\in L^1(\Pi)$.

\end{proof}

\subsection{Hamburger and Krein entire functions}\label{HK}
The Hamburger class of entire functions consists of all transcendental (non-polynomial) entire functions $F$ of exponential type zero,
that are real on $\R$, have only real simple zeros $\{\lan\}\subset\R$ and satisfy
$$\lim_{|n|\to\infty}\frac{|\lan|^a}{|F'(\lan)|}=0$$
for all $a>0$. If instead of the last equation the derivatives of $F$ satisfy
$$\sum\frac{1}{|F'(\lan)|}<\infty,$$
Then $F$ is said to belong to the Krein class of zero-type entire functions. Since zero sets of entire functions of zero exponential type
have zero denisty,
the Krein class contains the Hamburger class.

\ms\no Both classes play important roles in  approximation problems, see  \cite{BS} for further references.

\ms\no We say that $\L=\{\lan\}$ is a zero set of $F$ if $\{F=0\}=\L$. Our methods give the following description
of zero sets of Hamburger and Krein functions.
$$$$
\begin{proposition}$$$$
I) A discrete sequence $\L=\{\lan\}\subset\R$ is a zero set of a Hamburger entire function if and only if
$\L$ is a balanced zero density sequence whose characteristic sequence $P=\{p_n\}$ satisfies
$$\lim_{|n|\to\infty}\frac{\log|\lan|}{p_n}=0.$$
If $\L$ is such a sequence then there exists a unique up to a constant multiple Hamburger entire function $F$
with the zero set $\L$. The function $F$ is given by the formula
$$F=\frac {\const}{K\mu},$$
where $\mu$ is a finite discrete measure concentrated on $\L$,
$$\mu=\sum (-1)^n\exp(p_n)\delta_{\lan}.$$

\ms\no II) A discrete sequence $\L=\{\lan\}\subset\R$ is a zero set of a Krein entire function of exponential type zero if and only if
$\L$ is a balanced zero density sequence whose characteristic sequence $P=\{p_n\}$ satisfies
$$\sum\exp(p_n)<\infty.$$
If $\L$ is such a sequence then there exists a unique up to a constant multiple zero type Krein entire function $F$
with the zero set $\L$. The function $F$ is given by the formula
$$F=\frac {\const}{K\mu},$$
where $\mu$ is a finite discrete measure concentrated on $\L$,
$$\mu=\sum (-1)^n\exp(p_n)\delta_{\lan}.$$

\end{proposition}

\begin{proof}
I) If $F$ is a Hamburger function then one can consider a measure $\mu$ concentrated on $\L$, $\mu(\{\lan\})=1/F'(\lan)$.
By noting that $1/F$ and $K\mu$ have the same residues at $\L$, we conclude that $1/F-K\mu$ is an entire function of zero
exponential type that tends to zero along the imaginary axis. Hence $F=K\mu$. The rest of the statement follows from
lemma \ref{masses} and the proof of theorem \ref{main}.

\ms\no II) can be established in a similar way.
\end{proof}

\subsection{A result by Borichev and Sodin}\label{sectionBS} One of the main results of the well-known paper by Borichev and Sodin, devoted to the so-called Hamburger moment problem, is the following
theorem on density of polynomials in $L^p(\mu)$, where $\mu$ is a measure concentrated on a zero set of
a Hamburger function.

\begin{theorem}\label{BStheorem}
Let $\L=\{\lan\}$ be a zero set of a Hamburger function and let $\mu=\sum\alpha_n\delta_{\lan}$ be a finite positive measure.
If $1<p<\infty$ then
polynomials are dense in $L^p(\mu)$ if and only if
for any Hamburger function $F$, such that $\{F=0\}=\G\subset\L$,
$$\sum_{\lan\in\G}\left[\frac 1{\alpha_n^{1/p}|F'(\lan)|}\right]^{\frac p{p-1}}=\infty.$$
Polynomials are dense in $L^1(\mu)$ if and only if
for any Hamburger function $F$, such that $\{F=0\}=\G\subset\L$,
$$\liminf_{|\lan|\to\infty,\ \lan\in\G}  \alpha_n F'(\lan)=0.$$
Polynomials are dense in $C_W$, where $W(\lan)=1/\alpha_n$ and $W\equiv\infty$ on $\R\setminus\L$, if and only if
for any Hamburger function $F$, such that $\{F=0\}=\G\subset\L$,
$$\sum_{\lan\in\G}\frac 1{\alpha_n|F'(\lan)|}=\infty.$$
\end{theorem}

\begin{proof}
By the last proposition, for any Hamburger function $F$, the function $1/F$ is a Cauchy integral
of a finite measure $\mu=\const\sum (-1)^n\exp(p_n)\delta_{\lan}$. Since $F$ is transcendental, it has to grow sup-polynomially along $i\R$. Hence by lemma \ref{growth}
$\mu$ annihilates polynomials. Now the theorem follows from corollary \ref{discrete}.
\end{proof}

\ms\no Note that the condition that $\L$ is a zero set of a Hamburger function can be dropped from the statement
of theorem \ref{BStheorem}.

\subsection{Asymptotics of characteristic sequences and applications}\label{charass}
Let $u$ be a monotone increasing function on $\R$. 
Suppose that the harmonic conjugate function $\tilde u$ is Poisson-summable, i.e. $\tilde u\in L^1(\Pi)$.
Let $\L=\{\lan\}$ be a sequence such that
$u(\lan)=n\pi$.

\ms\no It is not difficult to show that then $\L$ is a zero density balanced sequence.
(This condition is actually equivalent to $\tilde u\in L^1(\Pi)$.)
Let $P=\{p_n\}$ be the characteristic sequence
of $\L$.

\ms\no Elementary estimates yield:

\begin{proposition}\label{conjugate}
Suppose that $u'(x)$ exists and is bounded for large enough $|x|$. Then 
$$p_n= \ti u(\lan) +O(\log |\lan|)$$
as $|n|\to\infty$.
\end{proposition}



\ms\no Remark \ref{x^n} together with theorem \ref{main} give the following

\bs
\begin{corollary}
$$$$
I) If $W$ is a regular weight such that $\log W(\lan)\leqslant \ti u(\lan) + O(\log |\lan|)$ then polynomials are not dense in $C_W$.

\bs\no II) If $\mu=\sum \alpha_n\delta_{\lan}$ is a finite positive measure such that
$$\sum \alpha_n^{1-q}\exp{ q p_n}<\infty$$
for some $1<q<\infty$ then polynomials are not dense in $L^p(\mu), \frac 1p+\frac 1q=1$.

\bs\no III) If
$$\alpha_n=O(\exp{p_n})$$
then polynomials are not dense in $L^1(\mu)$.
\end{corollary}

\ms\no For many examples of discrete sequences $\L$ one can easily find a suitable function $u$ and the values of its conjugate
at $\L$.
If, for instance, $\L=\{n^{1/\alpha}\}_{n\geqslant 0},  \ 0<\alpha<1/2$ then one may consider $u$ defined as
$$u(x)=\begin{cases}\pi x^\alpha\textrm{ if }x\in\R_+\\
0\textrm{ if }x\in\R_-
\end{cases}$$
  and
find that
$$\ti u(n^{1/\alpha})=-\pi n\tan\left(\alpha\pi-\frac \pi2\right).$$
In the two-sided case $\L=\{\pm n^{1/\alpha}\}_{n\geqslant 0}, \ 0<\alpha<1$,  one may use $u$ defined as
$$u(x)=\begin{cases}\pi x^\alpha\textrm{ if }x\in\R_+\\
-\pi|x|^\alpha\textrm{ if }x\in\R_-
\end{cases}.$$
Then
$$\ti u(\pm n^{1/\alpha})=-\pi n\tan\left(\alpha\frac\pi 2-\frac\pi 2\right).$$
Such simple calculations and estimates,
together with statements from this section, yield majority of the examples
of discrete measures, whose $L^p$ spaces are not spanned by polynomials, existing in the literature.

\end{document}